\documentclass[12pt]{article}
\usepackage{amsmath, amsthm, amssymb,amsfonts,latexsym}
\usepackage{graphicx}
\usepackage[initials]{amsrefs}
\theoremstyle{plain}
\newtheorem{thm}{Theorem}
\newtheorem{prop}[thm]{Proposition}

\theoremstyle{definition}

\theoremstyle{remark}
\newtheorem{rem}[thm]{Remark}

\def\R{ {\mathbf R} }

\def\Ex{ \mathbf{E} }

\numberwithin{equation}{section}

\begin{document}

\title{{\bf Tau functions of KP solitons realized in Wiener space}}
\author{Hidemi Aihara, Jir\^o Akahori, Hiroko Fujii and Yasufumi Nitta \\
Ritsumeikan University}
\date{July, 2011}
\maketitle
\begin{abstract}
In this paper, a probabilistic representation of 
the tau functions of KP (Kadomtsev-Petviashvili) 
solitons in terms of 
stochastic areas will be presented. 
\end{abstract}
\section{Introduction}
In the introduction, after giving a very short introduction to
the theory of solitons following \cite{MR1736222}, 
we recall some existing results 
from probabilistic approaches. 

\subsection{Solitons, tau-functions, and Sato's Grassmannian}\label{Sato}
By {\em solitons}, we usually mean {\em solitary wave}
solutions (behaving like a particle) 
to a class of non-linear wave equations including the KdV 
(Korteweg-de Vries) equation 
\begin{equation}\label{KdV}
\frac{\partial u}{\partial t}= \frac{1}{4} \frac{\partial^3 u}{\partial x^3} 
+ \frac{3}{2} u \frac{\partial u}{\partial x}
\end{equation}
as the most notable example. 


The first giant step in the study of solitons 
was made by Gardner, Greene, Kruskal and Miura \cite{GGKM}, 
where they observed that 
(i) 
the eigenvalues of 
the Shor\"odinger operator 
\begin{equation*}
\frac{\partial^2}{\partial x^2} + u (t,x),
\end{equation*}
where $ u $ is a solution to (\ref{KdV}),
is constant in time parameter $ t $, 
and (ii) one can construct a soliton solution to 
(\ref{KdV}) by applying the 
{\em inverse scattering method}, by which we mean the (mathematical)
method to construct (unknown) potentials out of given 
{\em scattering data},  
which had already been fully developed. 
The relation is most clearly seen when the potential is 
{\em reflectionless} as
\begin{equation}\label{ISM}
u (t,x) = 2\frac{d^2}{dx^2}\log \det( I + G(x,t)),
\end{equation}
where 
\begin{equation*}
G (x,t)
:= \left( \frac{\sqrt{m_{i}m_{j}}
e^{(\eta _{i}+\eta _{j})x + (\eta_i^3 + \eta_j^3) t}}{\eta _{i}+\eta _{j}} \right)_{1\leq i,j\leq n}.
\end{equation*}
The constants $ \eta_j, m_j $, $ j=1, \cdots, n $ are so-called
{\em scattering data}.  

The observation (i) together with the awareness of the existence
of the infinite invariants in \cite{GGKM} motivated another 
seminal paper by P. Lax \cite{MR0235310}, where 
the KdV equation (\ref{KdV}) is understood as the compatibility
between the two equations:
\begin{equation*}
\begin{cases}
\left(\frac{\partial^2}{\partial x^2} + u (t,x)\right)w (=:P w) = \kappa w, 
& \text{($ \kappa $ is an eigenvalue)}\\
\left( \frac{\partial^3}{\partial x^3} 
+ \frac{3}{2}u \frac{\partial}{\partial x}
+ \frac{3}{4} \frac{\partial u}{\partial x} 
\right) w (=: B w) = 0.
\end{cases}
\end{equation*}
This compatibility is rephrased as the celebrated ``Lax equation":
\begin{equation}\label{LaxKdV}
\frac{\partial P}{\partial t} + [P, B] = 0,
\end{equation}
where the bracket is the commutator; $ [P, B] = PB - BP $. 

By considering pseudo differential operators such as $ \partial^{-n} $
for $ n \in \mathbf{N} $ and their infinite series, 
we have in fact $ B = (P^{3/2})_+ $, 
where $ ( D )_+ $ is the differential operator part of the pseudo differential operator $ D $. In this Lax form, 
the existence of the infinite many invariants can be rephrased as 
\begin{equation*}
\frac{\partial P}{\partial x_k} + [P, (P^{k/2})_+] = 0, \quad k=1, 3, 5, \cdots,2n+1,\cdots,
\end{equation*}
where $ u \equiv u (x_1, x_3, \cdots, x_{2n+1}, \cdots ) $, a function 
of infinitely many variables. The KdV case (\ref{LaxKdV}) 
is retrieved by setting $ x_1= t, x_3 = x $. 
Each Lax equation generates a non-linear evolution equation with respect to
$ x_{2k+1} $ since 
$ [P, (P^{k/2})_+] $'s are all 
multiplication operators.  The totality of the generated equations 
is usually called {\em KdV hierarchy}. 

If we instead start with the operator 
\begin{equation*}
L = \partial + \sum_{j=1}^\infty u_j \partial^{-j}, 
\end{equation*}
then we still have that $ [L, (L^k)_+] $ are all multiplication operators, 
and hence we obtain infinitely many nonlinear differential equations 
with respect to $ u_j $'s of infinitely many variables $ x_1, x_2, \cdots, x_n,
\cdots $
by the Lax equations:
\begin{equation*}
\frac{\partial L}{\partial x_k} + [L, (L^{k})_+] = 0, \quad k=1, 2, \cdots.
\end{equation*}
The family is called {\em KP hierarchy} since the KP (Kadomtsev-Petviashvili) 
equation,
\begin{equation*}
\frac{3}{4}\frac{\partial^2 u_1}{\partial x_2^2}
= \frac{\partial}{\partial x_1} \left( \frac{\partial u_1}{\partial x_3} 
- \frac{3}{2} u_1 \frac{\partial u_1}{\partial x_1}
- \frac{1}{4} \frac{\partial^3 u_1}{\partial x_1^3} \right),
\end{equation*}
which is easily seen to be a generalization of the KdV to 
a two dimensional model, 
is deduced from the equations with $ k =2 $ and $ k=3 $. 
The KP hierarchy as a whole 
is also a generalization of the KdV hierarchy since 
the latter hierarchy is obtained
by a reduction $ ( L^2 )_- = 0 $ from the former. 

The equations in KP/KdV hierarchy are all ``soliton equations" 
in the sense that they all have exact solutions of soliton type
\footnote{The solitons are basically rational functions of 
the exponential functions of the form $ e^{\sum c_{ij} x_j} $
for some constants $ c_{ij} $'s.}. 
In fact, according to Sato's 
theory of infinite dimensional Grassmannian (\cite{MR730247}, see also \cite{MR1013125,MR1736222}), 
all the $ u_j $'s of the hierarchy are simultaneously generated from 
a single function called {\em tau-function} $ \tau $ in the following way:
determine $ w_1 $, $ w_2 $, etc, by 
\begin{equation}\label{waveft}
\frac{\tau (x_1-\frac{1}{k}, x_2- \frac{1}{2k^2}, \cdots)}
{\tau(x_1, x_2, \cdots)}
= 1 + \frac{w_1}{k} + \frac{w_2}{k^2}+ \cdots
\end{equation}
by comparing the coefficients of $ k^{-j} $, $ j=1,2, \cdots $,
and then $ u_1, u_2 $, etc by 
\begin{equation}\label{dmodule}
L=\left(1+\sum_{j=1}^\infty w_j \partial^{-j} \right) \circ \partial \circ
\left(1+\sum_{j=1}^\infty w_j \partial^{-j}
\right)^{-1}.
\end{equation}
For example, we have
\begin{equation}\label{U1}
u_1 = 2\frac{\partial^2}{\partial x_1^2} \log \tau. 
\end{equation}
In particular, we see that if $ \tau $ is a polynomial of $ e^{\sum c_{ij} x_j} $'s, then $ u_j $'s are all 
``solitons" in that they are all rational functions 
of $ e^{\sum c_{ij} x_j} $'s. 

The tau functions are characterized as a solution to 
a family of quadratic differential equations called {\em Hirota equations},
which are nothing but Pl\"ucker relations that define 
Sato's infinite dimensional  
Grassmannian. That is to say, a tau function of the KP hierarchy is a point in 
the Sato's Grassmannian.
It should be noted that in the Sato's theory, the KP hierarchy is the 
most universal one, out of which many well-known soliton equations are obtained by
a reduction. 

The following functions 
are known to be among the 
tau functions of the soliton solution of the KP equation:
\begin{equation*}
\begin{split}
& \tau (x_1, x_2, \cdots) \\
&= \sum_{J \subset I} \left(\prod_{ i \in J} m_i \right)
\left( \prod_{i,i' \in J, i<i'} \frac{(p_i-p_{i'})(q_i-q_{i'})}
{(p_i-q_{i'})(q_i-p_{i'})} \right)
\exp \left(\sum_{i \in J } \sum_{l=1}^\infty (p_i^l -q_i^l) x_l
\right)
\end{split}
\end{equation*}
for $ I = \{ 1, \cdots, n \} $, $ n \in \mathbf{N} $, 
where $ m_1, \cdots, m_n $, $ p_1, \cdots, p_n $, and 
$ q_1, \cdots, q_n $ are (indefinite) constants. This is alternatively
written as 
\begin{equation}\label{GISM}
\tau (x_1, x_2, \cdots) = \det (I + G (x_1, x_2, \cdots )),
\end{equation}
where 
\begin{equation*}
G (x_1, x_2, \cdots )= \left(\frac{\sqrt{m_im_j}}{p_i-q_j}
e^{\frac{1}{2} \sum_{l=1}^\infty \{(p_i^l -q_i^l)+ (p_j^l-q_j^l)\} x_l}  \right)_{i,j}.
\end{equation*}
The formula (\ref{GISM}) is a generalization of (\ref{ISM})
since we retrieve it by (\ref{U1}) 
and the reductions of $ q_j = -p_j $, $ x_l =0 $ for $ l \geq 4 $.

\begin{rem}\label{trivialfactor}
It should be noted that, if $ f $ is a solution to 
a Hirota equation 
then so is $ C e^{\sum_j c_j x_j} f $, 
for arbitrary constants $ C, c_1, c_2, \cdots $.
Therefore tau function is stable under the multiplication of 
the factor $ C e^{\sum_j c_j x_j} $. This property will be used
in the proof of Theorem \ref{KPS}. 
\end{rem}

\subsection{Probabilistic approach to solitons}
As far as we know, the first attempt to represent solitons
in terms of the expectation of Wiener functionals was made by
S. Kotani \cite{Kotani} in 2000. According to \cite{MR2083709}, 
Kotani constructed the following correspondences. 
Let $ \Sigma $ be the set of all pairs $ (\sigma_+, \sigma_-) \equiv \sigma $
of non-negative measures each on $ \mathbf{R}_- $ such that
$ \int_{\mathbf{R}_-} e^{\sqrt{-1}\lambda} \sigma_\pm (d \lambda) < \infty $
for any $ \lambda > 0 $. For $ \sigma \in \Sigma $, 
associate a Gaussian process
$ X^\sigma $ with mean $ 0 $ whose covariance $ C(u,v) = \Ex [X(u)X(v)] $
is given by 
\begin{equation*}
\begin{split}
&C(u,v; \sigma)
= \frac{1}{4} \int_{\mathbf{R}_-} (-z)^{-1/2} 
\left( e^{\sqrt{-z}(u+v)} - e^{\sqrt{-z}|u-v|}
\right) \sigma_+ (dz) \\
&\hspace{2cm}+ \frac{1}{4} \int_{\mathbf{R}_-} (-z)^{-1/2} 
\left( e^{-\sqrt{-z}|u-v|} - e^{-\sqrt{-z}(u+v)}
\right) \sigma_- (dz). 
\end{split}
\end{equation*}
Let $ \mathcal{Q} $ be the totality of the function $ q^\sigma $
with $ \sigma \in \Sigma $, where 
\begin{equation}\label{ISM-Gauss1}
q^\sigma (x) = -4 \frac{\partial^2}{\partial x^2} 
\log \Ex \left[\exp \left(-\frac{1}{2} \int_0^x |X^\sigma (y) |^2 ds
\right)
\right].
\end{equation}
Then Kotani showed that 
$ \mathcal{Q} $ is the closure (with respect to the topology 
of uniform convergence on compacts) of $ \cup \mathcal{Q}_n $, 
where $ \mathcal{Q}_n $ be totality of the reflectionless potentials 
of scattering data consisting of $ 2n $ constants. 

In a somewhat different line, K. Hara and N. Ikeda \cite{MR1869989} derived
from the Fourier transform of a class of quadratic Wiener functionals 
a dynamics in the Grassmannian as a finite dimensional analogue 
to the Sato's framework (\ref{waveft})-(\ref{dmodule}) etc. 

Soon after that N. Ikeda and S. Taniguchi \cite{MR2083709} 
obtained a specific and more ``stochastic analysis oriented" construction
of the Gaussian process $ X^\sigma $ in (\ref{ISM-Gauss1}) than 
Kotani's method.
They set 
\begin{equation}\label{TI1}
X^\sigma_t = \sqrt{a} \langle c, \xi^p_t \rangle,
\end{equation} where 
$ a > 0 $, 
$ c \in \mathbf{R}^n_+ $, $ p \in \mathbf{R}^n $
and $ \xi^p $ is an Ornstein-Uhlenbeck process
in $ \mathbf{R}^n $ starting at $ 0 $
defined as the soliton to the following SDE:
\begin{equation}\label{TI2}
d \xi_t = dW_t + \mathrm{diag}\{p_1, \cdots, p_n \} \xi_t dt.
\end{equation} 
The measure $ \sigma $ in Kotani's correspondence is given as
\begin{equation*}
\sigma_\pm  (du) = 2 a^2 \sum_{i: p_i \in \mathbf{R}_{\pm}}c^2_i 
\delta_{-p_i^2} (du). 
\end{equation*}
Related studies and surveys concerning the quadratic Wiener functionals 
can be found in \cite{MR1749870,MR2091362,MR2197979,MR2391539}, 
and more recently in \cite{MR2603056}.

Here we remark that 
all the probabilistic results cited here are on KdV solitons, and 
not extendable to KP. 
In this paper, we will present a probabilistic representation 
of KP solitons using generalized {\em stochastic areas} 
(see Theorem \ref{thm01} and Theorem \ref{KPS}). 

\subsection{Organization of the present paper}
In section \ref{GLSA}, we will introduce 
L\'evy's stochastic area formula
and present its generalization as Theorem \ref{thm01}
and its proof. Then in section \ref{PKP}, 
we will show that the generalized stochastic area formula 
is parameterized as a tau function of KP solitons. 
In section \ref{T-I}, we will give a probabilistic 
interpretation of the reduction from KP- to KdV-solitons. 
In section \ref{2D} we will present 
another realization, where
the dimension of the Wiener space is fixed to two.

\section{A generalization of L\'evy's stochastic area formula}\label{GLSA}
Let $ (\Omega, \mathcal{F}, P ) $ be a probability space 
and $ W \equiv (W^1, W^2 ) $ be a two-dimensional 
Brownian motion on it. 
The area enclosed by the curve $ s \mapsto W_s $ 
and its chord up to time $ t $,
which is usually called stochastic area of $ W $, 
is given by 
\begin{equation*}
S_t := \frac{1}{2}\left( \int_0^t W^2_s \,dW^1_s 
- \int_0^t W^1_s \,dW^2_s \right).
\end{equation*}
The characteristic function of $ S_t $
is explicitly given as 
\begin{equation}\label{Levy01}
\Ex [e^{\sqrt{-1} \xi S_t}]
= (\cosh \frac{\xi t}{2} )^{-1}  \quad (\xi \in \R), 
\end{equation}
and conditioned one is also given explicitly as
\begin{equation}\label{Levy02}
\Ex [e^{\sqrt{-1} \xi S_t}| W^1_t = x, W^2_t=y ]
= \frac{\xi t}{2 \sinh \frac{\xi t}{2} }e^{\frac{1}{2}
(x^2+y^2)(1-\frac{\xi t}{2}\coth \frac{\xi t}{2})}  \quad (\xi \in \R), 
\end{equation}
which were found by Paul L\'evy \cite{MR0044774}
using Fourier series expansion of $ W $. 
Either is often called L\'evy's (stochastic area) formula(s). 
There have been plenty of studies related to the formulas. 
For example, the heat kernel of the Heisenberg group can be
obtained by a slight modification of the formula 
(\cite{MR0461589}, see also \cite{IW}). 
Many alternative proofs and 
generalizations have been found 
(\cite{MR580140}, \cite{MR898500}, \cite{MR1007231}, 
\cite{MR1335477}, \cite{MR1266245}, 
etc).

In this paper, we give the following generalization of (\ref{Levy01}).
In its proof, the second L\'evy formula (\ref{Levy02}) plays a crucial role. 
\begin{thm}\label{thm01}
Let $W^l \equiv (W^{l,1}, W^{l,2}) $, $ l=1,\cdots,n $
be mutually independent two-dimensional 
Brownian motions starting at the origin,
and stochastic areas of $ W^l $ will be denoted by
\begin{eqnarray*}
S^l := \int^{1}_{0} \left( W^{l,2}_{s} dW^{l,1}_{s}
                   - W^{l,1}_{s} dW^{l,2}_{s} \right).
\end{eqnarray*}
Let $ \Lambda := \mathrm{diag} \{ \lambda_1, \cdots, \lambda_n \} $,
where $ \lambda_l, l=1,2, \cdots, n $ are positive numbers. 
Let $ A \equiv (a_{i,j})_{1 \leq i,j \leq n} $ be a real $ n \times n $ matrix, and 
$ C^{\pm} $ be its symmetric and skew-symmetric part respectively, 
namely, $ C^{\pm} = (A \pm A^*)/2 $. 
Denote 
$ \mathbf{W}^i_t = (W^{1,i}_{t} , \cdots ,W^{n,i}_{t}) 
$ for $ i=1,2 $, and define for $ z \in \mathbf{C} $
\begin{equation}\label{area01}
\begin{split}
\hat S (z) &\equiv 
\hat S_{A,\Lambda} (z) \\
& := z \sum^{n}_{l=1} {\lambda}_l S^l
+ z \langle \Lambda^{\frac{1}{2}} C^- \Lambda^{\frac{1}{2}} \mathbf{W}_1^1, \mathbf{W}_1^2 \rangle_{\mathbf{R}^{n}}
 - \frac{z^2}{2} \sum_{i=1,2} \langle
   \Lambda^{\frac{1}{2}} C^+ \Lambda^{\frac{1}{2}} \mathbf{W}_1^i, \mathbf{W}_1^i
   \rangle_{\mathbf{R}^{n}}.
\end{split}
\end{equation}
Then, if  either $ \max_l |\lambda_l| $ or $ \Vert C^+ \Vert $ 
is sufficiently small, we have 
\begin{equation*}
\begin{split}
& \Ex [ e^{ \hat S (\sqrt{-1}) }] \\
& = 
\begin{vmatrix}
\cosh \lambda_1 + a_{1,1} \sinh \lambda_1 &  a_{1,2} \sinh \lambda_2 
& \cdots & a_{1,n} \sinh \lambda_{n} \\
a_{2,1} \sinh \lambda_1 & \cosh \lambda_2 + a_{2,2} \sinh \lambda_2& \cdots 
& a_{2,n} \sinh \lambda_{n} \\
\vdots & \vdots & \ddots & \vdots \\
a_{n,1} \sinh \lambda_1 & a_{n,2} \sinh \lambda_2 & \cdots 
& \cosh \lambda_n + a_{n,n} \sinh \lambda_n 
\end{vmatrix}^{-1}.
\end{split}
\end{equation*}
\end{thm}


\begin{proof}
We first calculate the conditional expectation of $ e^{\hat S (\sigma)} $ 
conditioned by $ \mathbf{W}_1 = (\mathbf{W}^1_1, \mathbf{W}^{2}_1)$ .
By the L\'evy's formula (\ref{Levy02}) 
with some analytic continuation, 
we have for sufficiently small $ \sigma \in \mathbf{R} $
(such that the random variable $ e^{\sigma \sum \lambda_l S^l} $ is
integrable),  
\begin{equation}\label{LSAF2}
\begin{split}
&E[e^ { \sigma \sum_l \lambda_l S^l } | \mathbf{W}_1 ] \\
& =\prod_l \frac{ {\sigma \lambda}_l }{\sin {\sigma \lambda}_l }
                  \exp \Big( 
                           -\frac{(W^{l,1}_1)^2 + (W^{l,2}_1)^2}{2}
                          ( {\sigma \lambda}_l \cot{\sigma \lambda}_l  - 1)
                  \Big).
\end{split}
\end{equation}
Therefore we have
\begin{equation*}
E[e^{\hat S (\sigma)} | \mathbf{W}_1] 
= \prod_l \frac{ {\sigma \lambda}_l}{\sin {\sigma \lambda}_l}                \exp \Big(-\frac{1}{2}
                   \langle
                    (M(\sigma) - \mathbf{I} + C(\sigma)) \mathbf{W}_1,
                     \mathbf{W}_1
                   \rangle
\Big)
\end{equation*}
where 
\begin{equation*}
M (\sigma)  = 
\begin{pmatrix}
\sigma \Lambda \cot \sigma \Lambda & 0 \\
0 & \sigma \Lambda \cot \sigma \Lambda 
\end{pmatrix},
\end{equation*}
with
\begin{equation*}
\cot \sigma \Lambda:= \mathrm{diag} \{ \cot \sigma {\lambda}_1, \cdots,
\cot \sigma {\lambda}_n\}
\end{equation*}
as usual, and 
\begin{equation*}
C (\sigma) := 
\begin{pmatrix}
\quad \sigma^2 \Lambda^{\frac{1}{2}} C^+ \Lambda^{\frac{1}{2}} & \sigma \Lambda^{\frac{1}{2}} C^- \Lambda^{\frac{1}{2}} \\
-\sigma \Lambda^{\frac{1}{2}} C^- \Lambda^{\frac{1}{2}} & \sigma^2 \Lambda^{\frac{1}{2}} C^+ \Lambda^{\frac{1}{2}}
\end{pmatrix}. 
\end{equation*}

Since $ \Vert M (\sigma) + C (\sigma) -I \Vert \to 0 $ as 
$ \sigma \to 0 $, we can take $ \sigma $ small enough to ensure 
that $ M (\sigma ) + C (\sigma) $ is positive definite. 
Then, applying quadratic Gaussian formula for such $ \sigma $, we obtain 
\begin{equation}\label{area03}
\Ex [ e^{ \hat S (\sigma) }]
= \prod^n_{l=1} \frac{ \sigma {\lambda}_l}{\sin{\sigma \lambda}_l }
 \ \det (M (\sigma) + C(\sigma) )^{-\frac{1}{2}}.
\end{equation}
We may go further as 
\begin{eqnarray*}
&&\hspace{-1cm} \det (M (\sigma) + C(\sigma)) \\
&&\hspace{-1cm} = \det \left(
          \begin{array}{ccc}
          \sigma \Lambda^{\frac{1}{2}} ( \Lambda^{\frac{1}{2}} \cot \sigma \Lambda + \sigma C^+ \Lambda^{\frac{1}{2}}) 
& \sigma \Lambda^{\frac{1}{2}} C^- \Lambda^{\frac{1}{2}} \\
- \sigma \Lambda^{\frac{1}{2}} C^- \Lambda^{\frac{1}{2}} & \sigma \Lambda^{\frac{1}{2}} (\Lambda^{\frac{1}{2}} \cot \sigma \Lambda + \sigma C^+ \Lambda^{\frac{1}{2}})  \\
          \end{array}
          \right) \\
&&\hspace{-1cm} = \det (\sigma \Lambda^{\frac{1}{2}} ( \cot \sigma \Lambda +
\sigma C^+ + \sqrt{-1}C^-) \Lambda^{\frac{1}{2}}) \det (\sigma \Lambda^{\frac{1}{2}} ( \cot \sigma \Lambda + 
\sigma C^+ - \sqrt{-1}C^-) \Lambda^{\frac{1}{2}}) \\
&&\hspace{-1cm} (\text{Since $ C^- $ is skew symmetric}) \\
&&\hspace{-1cm} =  
\left\{\prod_{l} (\sigma \lambda_l)
\det  (\cot \sigma \Lambda + \sigma C^+ + \sqrt{-1} C^-)  \right\}^2.
\end{eqnarray*}
Hence (\ref {area03}) is turned into the following equality:
\begin{eqnarray}\label{area04}
&& \Ex [ e^{ \hat S (\sigma) }] 
= \det \Big( \cos \sigma \Lambda + 
(\sigma C^+ + \sqrt{-1} C^-) \sin \sigma \Lambda \Big)^{-1}
\end{eqnarray}
where $ 
\sin \sigma \Lambda:= \mathrm{diag} \{ \sin \sigma {\lambda}_1, \cdots,
\sin \sigma {\lambda}_n\}
$. 

The right-hand-side of (\ref{area04}) is apparently meromorphic 
in $ \sigma \in \mathbf{C} $.  
Now, we want to see if an analytic continuation to a domain 
including $ z= \sqrt{-1} $ is possible or not. 
To see this, it suffices to check the differentiability 
of the left-hand-side of (\ref{area04})
with respect to $ \sigma $. Namely, we need to check the integrability of
\begin{equation*}
\begin{split}
&\Ex [\frac{d}{d z} e^{\hat S (z)} ] \\
& \hspace{-.2cm}= \Ex \left[ e^{\hat S (z)} \left( 
\sum^{n}_{l=1} {\lambda}_l S^l
+ \langle \Lambda^{\frac{1}{2}} C^- \Lambda^{\frac{1}{2}} \mathbf{W}_1^1, 
\mathbf{W}_1^2 \rangle_{\mathbf{R}^{n}}
 - {z} \sum_{i=1,2} \langle
   \Lambda^{\frac{1}{2}} C^+ \Lambda^{\frac{1}{2}} \mathbf{W}_1^i, 
\mathbf{W}_1^i \rangle_{\mathbf{R}^{n}}
\right)
\right].
\end{split}
\end{equation*}
Since $ \hat{S} $ is quadratic Gaussian, the integrability 
is inherited from that of $ e^{\hat S (z)} $ itself, 
which is guaranteed if either $ \max_l |\lambda_l| $ or $ \Vert C^+ \Vert $ 
is sufficiently small. 
\end{proof}

\section{Parameterization as a tau function of KP solitons}\label{PKP}
As we have stated, a tau function $\tau$ of 
the $ n $-soliton solution of the Kadomtsev-Petviashvili equation 
(KP equation) is expressed by 
\begin{equation}\label{GISM2}
\tau (x_1, x_2, \cdots) = \det (I + G (x_1, x_2, \cdots )),
\end{equation}
with 
\begin{equation*}
G (x_1, x_2, \cdots )= \left(\frac{\sqrt{m_im_j}}{p_i-q_j}
e^{\frac{1}{2} (\xi_i + \xi_j )}  \right)_{1 \leq i,j \leq n},
\end{equation*}
where
\begin{eqnarray*}
& \xi_i = (p_i-q_i) x_1 + (p^2_i-q^2_i) x_2 + \cdots, \ \ i = 1, \cdots, n ,\\
\end{eqnarray*}
and $ m_i>0 $, $ p_i $ and $q_i$ are parameters.

\begin{thm}\label{KPS}
Let $ P = \Big( \frac{1}{p_i - q_j} \Big)_{1 \leq i,j \leq n}$, 
and assume that $ \min_{i,j} |p_i-q_j| $ is sufficiently large so that 
$ I+P $ is invertible. 
Then, if we put $ A = (I-P)(I+P)^{-1}$ 
and $ \Lambda := \mathrm{diag} \{ - \frac{1}{2} (\xi_1 + \log m_1),
\cdots, - \frac{1}{2} (\xi_n + \log m_n) \} $, we have that 
$ (\Ex [e^{\hat S_{A,\Lambda} (\sqrt{-1})}])^{-1} $, 
where $ \hat S_{A,\Lambda} $ is defined by (\ref{area01}),
defines a tau function of KP solitons. 
\end{thm}
\begin{proof}
Since 
\begin{equation*}
G = e^{-\Lambda} P e^{-\Lambda}, 
\end{equation*}
we have 
\begin{equation*}
\begin{split}
\tau &= \det ( I + e^{-\Lambda}  P e^{-\Lambda} ) \\
&= \det e^{-\Lambda}  \det ( e^{\Lambda}  + P e^{-\Lambda} ) 
= \det ( I + P e^{-2\Lambda} ).
\end{split}
\end{equation*}
On the other hand, 
\begin{equation*}
\begin{split}
& \det (\cosh \Lambda + A \sinh \Lambda) \\
& = \det \left( \frac{e^{\Lambda} + e^{-\Lambda}}{2} + A 
\frac{e^{\Lambda} - e^{-\Lambda}}{2} \right) \\
& = 2^{-n} \det \left\{ (I+A)e^{\Lambda}+(I-A) e^{-\Lambda} 
\right\} \\
& = 2^{-n} \det \{ (I + A) e^{\Lambda}\}
           \det \Big( I + (I + A)^{-1}(I - A) 
e^{-2 \Lambda}) \Big) \\
& = 2^{-n} \det (I + A) e^{-\frac{1}{2} \sum (\xi_i + \log m_i)}
\det \Big( I + P e^{-2 \Lambda} \Big).
\end{split}
\end{equation*}
The last equality follows since $$ A = (I-P)(I+P)^{-1} 
\iff P=(I + A)^{-1}(I - A). $$
As we have stated in Remark \ref{trivialfactor}, 
$ 2^{-n} \det (I + A) e^{-\frac{1}{2} (\xi_i + \log m_i) } $ 
is a trivial factor 
and thus by Theorem \ref{thm01} we have the assertion. 
\end{proof}

\section{Reduction to Ikeda-Taniguchi's construction}\label{T-I}
As we have discussed in section \ref{Sato}, 
we have (\ref{ISM}) by the reduction of $ q_j =-p_j $ in (\ref{GISM}).
In this section, we review this from the perspective of 
stochastic analysis. We will show that 
when $ C^- = 0 $, the expectation of the exponential of 
the generalized stochastic area is reduced to 
that of the exponential of the time integral of an 
Ornstein-Uhlenbeck process, which corresponds to the Taniguchi-Ikeda's
construction (\ref{ISM-Gauss1}), (\ref{TI1}) and (\ref{TI2})   
of reflectionless potentials/tau functions of KdV solitons. 

Precisely speaking, we have the following
\begin{prop}
Suppose that $ A $ in Theorem \ref{thm01} is symmetric. Then
\begin{equation*}
\Ex [e^{\hat S_{A,\Lambda}(\sqrt{-1})}] 
= \left(\Ex [e^{-\int_0^1 X^{A, \Lambda}_s ds}]
\right)^2 e^{\mathrm{tr} \Lambda A},
\end{equation*}
where $ X^{A, \Lambda} = \langle (\Lambda- A \Lambda A) \xi, \xi \rangle $ 
and $ \xi $ 
is an Ornstein-Uhlenbeck process on $ \mathbf{R}^d $ starting at $ 0 $
and satisfying
\begin{equation}\label{OUn}
d\xi_t = \Lambda^{\frac{1}{2}} dB_t + \Lambda A \xi_t \,dt,
\end{equation}
with $ B $ being an $ n $-dimensional standard Brownian motion.
\end{prop}

\begin{proof}
We first note the following identity since its right-hand-side 
also equals to that of (\ref{LSAF2}) with $ \sigma $ replaced by $ \sqrt{-1} $
(see e.g. \cite{MR2454984}):
\begin{equation*}
\Ex [e^ { \sqrt{-1} \sum_l \lambda_l S^l } | \mathbf{W}_1]= 
\Ex [e^{-\sum_l \frac{\lambda_l^2}{2} 
\int_0^1 \{(W^{l,1}_s)^2 + (W^{l,2}_s)^2\}\,ds} | \mathbf{W}_1  ].
\end{equation*}
Then since $ C^+ = A $ and $ C^- = 0 $, we have
\begin{equation*}
\begin{split}
\Ex [e^{\hat S_{A,\Lambda}(\sqrt{-1})}] 
&= \prod_{i=1,2} \Ex [e^{-\sum_l \frac{\lambda_l^2}{2} 
\int_0^1 (W^{l,i}_s)^2 \,ds 
+ \frac{1}{2} \langle \Lambda^{\frac{1}{2}}  A \Lambda^{\frac{1}{2}} 
\mathbf{W}^i_1,\mathbf{W}^i_1 \rangle }] \\
&=\left( \Ex [e^{-\sum_l \frac{\lambda_l^2}{2} 
\int_0^1 (W^{l,1}_s)^2 \,ds 
+ \frac{1}{2} \langle \Lambda^{\frac{1}{2}} 
A \Lambda^{\frac{1}{2}} \mathbf{W}^1_1,\mathbf{W}^1_1 \rangle }]
\right)^2.
\end{split}
\end{equation*}
By applying It\^o's formula, 
\begin{equation*}
\begin{split}
& e^{-\sum_l \frac{\lambda_l^2}{2} 
\int_0^1 (W^{l,1}_s)^2 \,ds +
\frac{1}{2} \langle \Lambda^{\frac{1}{2}} 
A \Lambda^{\frac{1}{2}} \mathbf{W}^1_1,\mathbf{W}^1_1 \rangle } \\
&= e^{\frac{1}{2} \mathrm{tr} \Lambda A} 
e^{\int_0^1 \langle \Lambda^{\frac{1}{2}} 
A \Lambda^{\frac{1}{2}} \mathbf{W}^1_s,d \mathbf{W}^1_s \rangle 
- \frac{1}{2} \int_0^1 | \Lambda^{\frac{1}{2}} A \Lambda^{\frac{1}{2}}  \mathbf{W}^1_s |^2 \,ds}
e^{- \frac{1}{2} 
\int_0^1 \langle (\Lambda - A \Lambda A) 
\Lambda^{\frac{1}{2}}\mathbf{W}^1_s, \Lambda^{\frac{1}{2}}
\mathbf{W}^1_s \rangle \,ds}.
\end{split}
\end{equation*}
Define $ Q $ by
\begin{equation*}
\frac{dQ}{dP} \bigg|_{\mathcal{F}_1} 
= e^{\int_0^1 \langle \Lambda^{\frac{1}{2}} 
A \Lambda^{\frac{1}{2}} \mathbf{W}^1_s,d \mathbf{W}^1_s \rangle 
- \frac{1}{2} \int_0^1 | \Lambda^{\frac{1}{2}} A \Lambda^{\frac{1}{2}}  \mathbf{W}^1_s |^2 \,ds}.
\end{equation*}
Then by the Maruyama-Girsanov theorem, we see that
$ \mathbf{W} $ under $ Q $ has the same law
as $ \xi $ of (\ref{OUn}). This completes the proof. 
\end{proof}

\section{A realization in 2D-Wiener space}\label{2D}
In this section, we shall show that
the $ 2n $-dimensional Brownian motion used 
to represent $ n $-solitons in Theorem \ref{thm01}
can be replaced by a $ 2 $-dimensional one irrespective of $ n $. 
This observation suggests that the whole Sato's Grassmannian would be
realized in 2-dimensional Wiener space. 

Let $W$ $\equiv$ $(W^1,W^2)$ be a $ 2 $-dimensional Brownian motion 
starting at the origin, 
$ \{ e^i = (e^i_1,\cdots e^i_n) , i =1, \cdots n \} $ 
be an orthonormal basis of $ \R^n $, 
and set 
\begin{equation*}
f_i (t) := \sqrt{n} \sum_{l=1}^n e^i_l 1_{[\frac{l-1}{n}, \frac{l}{n})}(t),
\quad i=1,2, \cdots, n.
\end{equation*}  
Define 
\begin{equation*}
S_{i,j}^+ 
:= \sum_{a=1,2} \left( \int_0^1 f_i(t) \,dW^a_t\right)
\left( \int_0^1 f_j (t) \,dW^a_t\right),
\end{equation*}
and
\begin{equation*}
S_{i, j}^- 
:=  \int_0^1 \left(\int_0^t f_j (s) \,d W^2_s \right) f_i(t) \,dW^1_t 
- \int_0^1 \left(\int_0^t f_i (s) \,dW^1_s  \right) f_j (t) \,dW^2_t. 
\end{equation*}

In this section we assume that $ \lambda_i > 0 $ for all $ i $. 
We shall denote the $ (i,j) $ entry of the matrices 
$ \Lambda^{\frac{1}{2}} C^+ \Lambda^{\frac{1}{2}} $ and
$ \Lambda^{\frac{1}{2}} (I+ C^-) \Lambda^{\frac{1}{2}} $
by $ \lambda^+_{i,j} $,
and $ \lambda^-_{i,j} $, respectively. 
Note that $ \lambda^-_{ii} = \lambda_i $. 
We assume that either $ \max_l |\lambda_l| $ or $ \Vert C^+ \Vert $ 
is sufficiently small to ensure the integrability. 
\begin{prop}\label{F1}
We have that
\begin{equation*}
\begin{split}
\Ex[e^{ \sum_{i,j} \left(\sqrt{-1} \lambda^-_{i,j} S^-_{i,j}
+\frac{1}{2} \lambda^{+}_{i,j} 
S^+_{i,j} \right)}]
&(= \Ex [e^{\hat S_{A, \Lambda} (\sqrt{-1})}]) \\
&= \det (\cosh \Lambda + A \sinh \Lambda)^{-1}.
\end{split}
\end{equation*}
\end{prop}
\begin{proof}
We will show the following equivalence in law:
\begin{equation*}
\begin{split}
& ( \sum_{i,j} \lambda^-_{i,j} S^-_{i,j}, \sum_{i,j} 
\lambda^{+}_{i,j} 
S^+_{i,j}) \\
&\overset{\mathrm{d}}{=}
\left( \sum^{n}_{l=1} {\lambda}_l S^l
+ \langle \Lambda^{\frac{1}{2}}
C^-\Lambda^{\frac{1}{2}}
\mathbf{W}_1^1, \mathbf{W}_1^2 \rangle_{\mathbf{R}^{n}}, 
\sum_{i=1,2} \langle \Lambda^{\frac{1}{2}}
C^+ \Lambda^{\frac{1}{2}} \mathbf{W}_1^i, \mathbf{W}_1^i
   \rangle_{\mathbf{R}^{n}} \right).
\end{split}
\end{equation*}
Let us start with the following direct calculation:
\begin{equation*}
\begin{split}
& S_{i,j}^+ = n \sum_{a=1,2}
\Bigl(
\sum_{k,l} e^i_{l} e^j_{k} (W^a_{k/n}-W^a_{(k-1)/n})(W^a_{l/n}-W^a_{(l-1)/n})
\Bigr),
\end{split}
\end{equation*}
and
\begin{equation*}
\begin{split}
& S_{i, j}^- = n \sum_{l=1}^{n} e_l^i e_l^j 
\int_{(l-1)/n}^{l/n} \left\{(W_{t}^2-W_{(l-1)/n}^2) dW_{t}^1 
- (W_{t}^1-W_{(l-1)/n}^1 ) dW_{t}^2 \right\} \\
& \hspace{2cm}+ n \sum_{k<l} 
\big\{ e_{k}^i e_{l}^j (W_{k/n}^2-W_{(k-1)/n}^2)
 (W_{l/n}^1-W_{(l-1)/n}^1)   \\
& \hspace{4cm} -  e^i_l e^j_k (W_{k/n}^1-W_{(k-1)/n}^1) 
(W_{l/n}^2-W_{(l-1)/n}^2) \big\}.
\end{split}
\end{equation*}
By the scaling property of the Brownian motion, the process
\begin{equation*}
\{ (W^1_{(s + l-1)/n}-W^1_{(l-1)/n}, 
W^2_{(s + l-1)/n}
- W^2_ {(l-1)/n})  : 0 \leq s \leq 1, l=1,\cdots,n \}
\end{equation*}
is identically distributed as
\begin{equation*}
\{ n^{-1/2}  
(W_s^{l,1}, W_s^{l,2} ): 0 \leq s \leq 1,l=1,\cdots,n \}
\end{equation*}
Here
$ \{ W^{l,1}, W^{l,2}, l = 1, \cdots, n\} $
are $2n$-dimensional Brownian motions
staring at the origin.
In particular, 
\begin{equation*}
\left\{ 
\int_{(l-1)/n}^{l/n} \{ (W_{t}^2-W_{(l-1)/n}^2) dW_{t}^1 
- (W_{t}^1-W_{(l-1)/n}^1 ) dW_{t}^2,  \}, l=1, \cdots, n \right\}
\end{equation*}
is identically distributed as
\begin{equation*}
\left\{ n^{-1} \int_{0}^1 
\left( W_{t}^{l,2} dW_{t}^{l,1} 
- W_{t}^{l,1} dW_{t}^{l,2} \right),
l=1, \cdots, n
\right\}, 
\end{equation*}
and 
\begin{equation*}
\left\{ 
(W_{k/n}^1-W_{(k-1)/n}^1) (W_{l/n}^2-W_{(l-1)/n}^2), 
1 \leq k,l \leq n
\right\}
\end{equation*}
is identically distributed as
$ \{ W^{k,1}_1 W^{l,2}_1 , 1 \leq k,l \leq n \}$.
Therefore, we have the following identity in law:
\begin{equation}\label{idl0}
\begin{split}
\sum_{i,j} \lambda^+_{i,j} S_{i, j}^+ 
&\overset{\mathrm{d}}{=} 
\sum_{a=1,2} 
\sum_{k,l} \sum_{i,j} \lambda^+_{i,j} e^i_l e^j_k W^{k,a}_1 W^{l,a}_1 \\
&= \sum_{a=1,2} \langle T^* \Lambda^{\frac{1}{2}} C^+
\Lambda^{\frac{1}{2}} T \mathbf{W}^a_1, \mathbf{W}^a_1 \rangle \\
& \overset{\mathrm{d}}{=} \sum_{a=1,2} \langle  \Lambda^{\frac{1}{2}} C^+
\Lambda^{\frac{1}{2}} \mathbf{W}^a_1, \mathbf{W}^a_1 \rangle
\end{split}
\end{equation}
and 
\begin{equation}\label{idl1}
\begin{split}
\sum_{i,j} \lambda^-_{i,j} S_{i, j}^- 
& \overset{\mathrm{d}}{=} 
\sum_{l=1}^{n} \sum_{i,j} \lambda^-_{i,j} 
e^i_l e^j_l  \int_{0}^{1} \left( W_{t}^{l,2} dW_{t}^{l,1} 
- W_{t}^{l,1} dW_{t}^{l,2} \right)  \\
& \quad + \sum_{k<l} \sum_{i,j} \lambda^-_{i,j} e^i_k e^j_l 
W^{k,2}_1 W^{l,1}_1 
- \sum_{i,j} \lambda^-_{i,j} e^i_l e^j_k W^{k,1}_1 W_1^{l,2} \\
& = \int_0^1 \langle T^*\Lambda T \mathbf{W}_{t}^{2}, 
d\mathbf{W}_{t}^1 \rangle
- \int_0^1 \langle T^*\Lambda T \mathbf{W}_{t}^{1}, 
d\mathbf{W}_{t}^2 \rangle \\
& \quad + \langle T^* \Lambda^{\frac{1}{2}} C^-
\Lambda^{\frac{1}{2}} T \mathbf{W}^1_1, \mathbf{W}^2_1 
\rangle \\
& \overset{\mathrm{d}}{=} 
 \sum_{l=1}^{n} \lambda_l \int_{0}^{1} \left( W_{t}^{l,2} dW_{t}^{l,1} 
- W_{t}^{l,1} dW_{t}^{l,2} \right)  \\
& \quad + \langle \Lambda^{\frac{1}{2}} C^-
\Lambda^{\frac{1}{2}} \mathbf{W}^1_1, \mathbf{W}^2_1 
\rangle.
\end{split}
\end{equation}
Here $ \mathbf{W}^a $, $ a=1, 2 $ are $ n $-dimensional Brownian motions,
and $ T:= [ e^1 \cdots e^n ]$, which is an orthogonal matrix
since $ \{e^1 \cdots e^n \} $ is an orthonormal basis.  
Note that $ \Lambda^{\frac{1}{2}} C^-
\Lambda^{\frac{1}{2}} $ is also skew-symmetric and thus 
$ (\Lambda^{\frac{1}{2}} C^-
\Lambda^{\frac{1}{2}})_{ii}=0 $ for $ i=1,\cdots, n $. 

Using the equivalences (\ref{idl0}) and (\ref{idl1}) in law, 
we can establish 
\begin{equation*}
\Ex[e^{ \sum_{i,j} \left(\sqrt{-1} \lambda^-_{i,j} S^-_{i,j}
+\frac{1}{2} \lambda^{+}_{i,j} 
S^+_{i,j} \right)}]
= \det \Lambda 
\det(\cosh \Lambda + \Lambda^{\frac{1}{2}} A \Lambda^{\frac{1}{2}} \sinh \Lambda)^{-1}
\end{equation*}
in the same manner as we have done in the proof of Theorem \ref{thm01}.
Since $ \Lambda^{\frac{1}{2}} $ and $ \sinh \Lambda $ commute, 
we have the assertion. 
\end{proof}

\end{document}